\documentclass[a4page]{article}
\usepackage[a4paper,
            bindingoffset=0.2in,
            left=1in,
            right=1in,
            top=1in,
            bottom=1in,
            footskip=.25in]{geometry}

\usepackage{amsmath}
\usepackage{graphicx}
\usepackage{xspace}
\usepackage{enumerate}
\usepackage[utf8]{inputenc}
\usepackage[square]{natbib}
\usepackage{color}
\usepackage[dvipsnames]{xcolor}
\definecolor{myblue}{rgb}{0, 0, 0.66}
\usepackage{url}
\usepackage{subfig}
\usepackage[colorlinks=true,linkcolor=blue,citecolor=myblue,urlcolor=cyan]{hyperref}
\usepackage{breakurl}
\usepackage{breakcites}

\usepackage{amsfonts}
\usepackage{amsthm}
\usepackage{bbm}
\usepackage{amssymb}   
\usepackage{mathabx}   
\usepackage{MnSymbol}
\usepackage{pifont}
\usepackage{cases}
\usepackage{mdframed}
\usepackage{stmaryrd}
\usepackage{mathpartir}
\usepackage{listings}
\usepackage{tikz}
\usepackage{tikz-cd}
\usetikzlibrary{decorations.pathmorphing}
\usepackage{multirow}
\usepackage{hhline}

\usepackage{amsmath}
\usepackage{graphicx}
\usepackage{xspace}
\usepackage{enumitem}
\usepackage[utf8]{inputenc}
\usepackage{color}
\usepackage{url}
\usepackage{subfig}
\usepackage{float}

\usepackage{amsfonts}
\usepackage{amsthm}
\usepackage{bbm}
\usepackage{MnSymbol}
\usepackage{pifont}
\usepackage{cases}
\usepackage{mdframed}
\usepackage{stmaryrd}
\usepackage{mathpartir}
\usepackage{listings}
\usepackage{tikz}
\usepackage{tikz-cd}
\usetikzlibrary{arrows,matrix,decorations.pathmorphing,
  decorations.markings, calc, backgrounds}
\usepackage[cmtip,all]{xy}
\newcommand{\longsquiggly}{\xymatrix{{}\ar@{~>}[r]&{}}}
\usepackage{mathtools}

\makeatletter
\DeclareRobustCommand\widecheck[1]{{\mathpalette\@widecheck{#1}}}
\def\@widecheck#1#2{%
    \setbox\z@\hbox{\m@th$#1#2$}%
    \setbox\tw@\hbox{\m@th$#1%
       \widehat{%
          \vrule\@width\z@\@height\ht\z@
          \vrule\@height\z@\@width\wd\z@}$}%
    \dp\tw@-\ht\z@
    \@tempdima\ht\z@ \advance\@tempdima2\ht\tw@ \divide\@tempdima\thr@@
    \setbox\tw@\hbox{%
       \raise\@tempdima\hbox{\scalebox{1}[-1]{\lower\@tempdima\box
\tw@}}}%
    {\ooalign{\box\tw@ \cr \box\z@}}}
\makeatother

\newtheorem*{heuristic}{Heuristic}
\newtheorem{lemma}{Lemma}
\newtheorem{definition}[lemma]{Definition}
\newtheorem{proposition}[lemma]{Proposition}

\newtheorem{therm}[lemma]{Theorem}
\newtheorem{corollary}[lemma]{Corollary}

\newcommand{\push}[1]{\normalfont{{\ensuremath{{\mathsf{push}}}(#1)}}}
\newcommand{\pushop}{\normalfont{\ensuremath{{\mathsf{push}}}}}
\newcommand{\pushl}[1]{\normalfont{\ensuremath{{\mathsf{push}}_{{\mathsf{l}}}}(#1)}}
\newcommand{\pushll}{\normalfont{\ensuremath{{\mathsf{push}}_{{\mathsf{l}}}}}}
\newcommand{\pushr}[1]{\normalfont{\ensuremath{{\mathsf{push}}_{{\mathsf{r}}}}(#1)}}
\newcommand{\pushlr}{\normalfont{\ensuremath{{\mathsf{push}}_{{\mathsf{lr}}}}}}
\newcommand{\pushrr}{\normalfont{\ensuremath{{\mathsf{push}}_{{\mathsf{r}}}}}}
\newcommand{\smashinc}[2]{\normalfont{\ensuremath{\langle #1,#2 \rangle}}}

\newcommand{\pushzero}[1]{\normalfont{\ensuremath{{\mathsf{push}}_{0}}(#1)}}
\newcommand{\pushone}[1]{\normalfont{\ensuremath{{\mathsf{push}}_{1}}(#1)}}
\newcommand{\pushtwo}[1]{\normalfont{\ensuremath{{\mathsf{push}}_{2}}(#1)}}
\newcommand{\pushzeroone}[1]{\normalfont{\ensuremath{{\mathsf{push}}_{0,1}}(#1)}}
\newcommand{\pushzerotwo}[1]{\normalfont{\ensuremath{{\mathsf{push}}_{0,2}}(#1)}}
\newcommand{\pushonetwo}[1]{\normalfont{\ensuremath{{\mathsf{push}}_{1,2}}(#1)}}

\newcommand{\teletype}[1]{\ensuremath{\mathtt{#1}}}
\newcommand{\systemname}[1]{\teletype{\color{darkgray}#1}\xspace}

\newcommand{\Agda}{\systemname{Agda}}

\newcommand{\agdaCubical}{\systemname{agda/cubical}}

\newcommand{\CubicalAgda}{\systemname{Cubical} \systemname{Agda}}

\newlength{\LETTERheight}
\AtBeginDocument{\settoheight{\LETTERheight}{I}}

\newcommand{\ox}{\otimes}%

\newcommand{\bZ}{\ensuremath{\mathbb{Z}}}

\newcommand{\refl}{\mathsf{refl}}

\newcommand{\trunc}[1]{\lvert\,#1\,\rvert}
\newcommand{\truncT}[1]{\lVert#1\rVert}

\newcommand{\north}{\mathsf{north}}
\newcommand{\south}{\mathsf{south}}

\newcommand{\merid}[1]{\mathsf{merid}\;{#1}}

\newcommand{\ap}[2]{\normalfont{\mathsf{ap}}_{#1}{(#2)}}
\newcommand{\apap}[2]{\ap{\mathsf{ap}_{#1}}{#2}}

\newcommand{\smashs}[2]{{#1}\,\widetilde{\wedge}\,{#2}}
\newcommand{\fw}[2]{\normalfont{\ensuremath{\mathsf{FW}_{#1}(#2)}}}
\newcommand{\fwred}[2]{\normalfont{\ensuremath{{\mathsf{FS}}_{#1}(#2)}}}

\definecolor{dkblue}{rgb}{0,0.1,0.5}
\definecolor{lightblue}{rgb}{0,0.5,0.5}
\definecolor{dkgreen}{rgb}{0,0.6,0}
\definecolor{dkbrown}{rgb}{0.4,0,0}
\definecolor{dkviolet}{rgb}{0.6,0,0.8}

\newcommand{\squigmapsto}{\resizebox{1mm}{1.5mm}{\raisebox{2mm}{$\vert$}}\hspace{-.6mm}$\rightsquigarrow$}

\newif\ifcitation
\citationtrue
\DeclareRobustCommand{\Vv}[1]{\ifcitation Van\else van\fi}

\begin{document}





\title{Symmetric Monoidal Smash Products in Homotopy Type Theory}

\author{
  Axel Ljungström\\
  \footnotesize{Stockholm University, Sweden}\\
  \footnotesize{\texttt{axel.ljungstrom@math.su.se}}
}
\date{}




\maketitle

\paragraph*{Abstract:}
In Homotopy Type Theory, few constructions have proved as troublesome
as the smash product. While its definition is just as direct as in
classical mathematics, one quickly realises that in order to define
and reason about functions over iterations of it, one has to verify an
exponentially growing number of coherences. This has led to crucial
results concerning smash products remaining open. One particularly
important such result is the fact that smash products form a
(1-coherent) symmetric monoidal product on the universe of pointed
types. This fact was used, without a complete proof, by
e.g. \citet{Brunerie16} in his PhD thesis to construct the cup product
on integral cohomology and is, more generally, a fundamental result in
traditional algebraic topology. In this paper, we salvage the
situation by introducing a simple informal heuristic for reasoning
about functions defined over iterated smash products. We then use the
heuristic to verify e.g. the hexagon and pentagon identities, thereby
obtaining a proof of symmetric monoidality. We also provide a formal
statement of the heuristic in terms of an induction principle
concerning the construction of homotopies of functions defined over
iterated smash products. The key results presented here have been
formalised in the proof assistant \CubicalAgda.

\paragraph*{Acknowledgements:}
The author would like to thank Evan Cavallo for several fruitful
discussions, for his prior unpublished formalisations of smash
products in \CubicalAgda---these have been useful as a source of
inspiration for the formalisation of this paper---and, of course, for
his contribution of the incredibly useful Lemma~\ref{lem:evan} to the homotopy type theory literature.

This paper is based upon research supported by the Swedish Research
Council (Vetenskapsrådet) under Grant No.~2019-04545.

\tableofcontents
\newpage



\section{Introduction}
\label{sec:intro}

In his 2016 proof of $\pi_4(S^3) \cong \bZ/2\bZ$ in Homotopy Type
Theory (HoTT), \citet{Brunerie16} crucially uses---but never proves in
detail---that the smash product is (1-coherent) symmetric
monoidal. Due to the vast amount of path algebra involved when
reasoning about smash products in HoTT, this has since remained open. While it
turns out that smash products are not needed for Brunerie's
proof~\citep{LM23}, the problem is still interesting in its own
right.

Several attempts have been made at salvaging the
situation. \citet{FlorisPhd} came very close by considering an
argument using closed monoidal categories but left a gap where the
path algebra became too technical. To be more precise, Van Doorn never
verified that the equivalence
$$(A \wedge B \to_\star C) \simeq (A \to_\star (B \to_\star C))$$ is a
pointed natural equivalence. Another line of attack by
Cavallo~and~Harper~\citep{CavalloHarper20, Cavallo21} is the addition
of parametricity to the type theory, which leads to a rather ingenious
proof of the theorem. This, of course, happens at the expense of making
the type theory more complicated. Yet another solution was studied by
\citet{brunerie18} who used \Agda meta-programming to generate
the relevant proofs. Possible philosophical objections to such a
solution aside, Brunerie's generated proof of the pentagon identity
failed to type-check due to its eating up too much memory.

In this paper, we provide another solution by introducing an informal
heuristic for reasoning about functions defined over iterated smash
products. This approach vastly reduces the complexity of the identity proofs involved. We
use this to give a complete proof of the fact that the smash product
is 1-coherent symmetric monoidal. We finally discuss how to make the heuristic formal
and express (a version of it) as a theorem which we also prove.

The paper is structured as follows. In Section~\ref{sec:background},
we introduce the two fundamental concepts of interest, namely
(1-coherent) symmetric monoidal (wild) categories and smash
products. Section~\ref{sec:assoc} introduces a new higher inductive
type (HIT) capturing double smash products. We use this to sketch the
construction of the associator map $\alpha_{A,B,C} : (A \wedge B)
\wedge C \xrightarrow{\sim} A \wedge (B \wedge C)$. In
Section~\ref{sec:heuristic}, we make a small detour and discuss
induction principles for the smash product. It is here that we
introduce the heuristic mentioned above. In
Section~\ref{sec:symmetric}, we apply our heuristic and sketch the
proof of the symmetric monoidality of the smash product. Finally, in
Section~\ref{sec:equaliser}, we discuss the translation of our heuristic
into a formal result. We provide one suggestion and prove it correct.

This paper is written in the informal flavour of type theory employed
in e.g. the HoTT Book \linebreak \citep{HoTT13}. Nevertheless, all key results have been
formalised in the proof assistant
\CubicalAgda~\citep{VezzosiMortbergAbel19}, a cubical extension of
\Agda which, in particular, enjoys native support for HITs. A file summarising the relevant formalisations can be found in
the \agdaCubical library at
\url{https://github.com/agda/cubical/blob/master/Cubical/Papers/SmashProducts.agda}
(and, alternatively, on a frozen branch at
\url{https://github.com/aljungstrom/cubical/blob/smashpaper/Cubical/Papers/SmashProducts.agda}).

\section{Background}
\label{sec:background}
Let us briefly introduce the key concepts of this paper: 1-coherent symmetric
monoidal (wild) categories and smash products. We assume familiarity
with HoTT and refer to the HoTT Book \linebreak \citep{HoTT13} for the basic
constructions and definitions used here.
\subsection{Symmetric monoidal wild categories}
To make the statements in this paper somewhat more compact, we will
employ the framework of \emph{wild categories}. These are defined to be just like categories but
without any h-level assumptions~\citep{CapriottiKraus17}:
\begin{definition}[Wild categories]
  A wild category is a category with a \textbf{type} of objects and
  \textbf{types} of morphisms.
\end{definition}
The difference between a wild category and a category is that the latter asks for \emph{sets} of morphisms. In this paper, the wild category of interest is that of pointed types.
\begin{proposition}\label{prop:pointed-precat}
  Let $\mathsf{Type}_\star$ denote the universe of pointed types (at
  some universe level). This universe forms a wild category with
  $\mathsf{Type}_\star[A,B] := (A \to_\star B)$, i.e.\ with pointed functions as morphisms.
\end{proposition}
The main goal of this paper is to show that $\mathsf{Type}_\star$ is not only a wild category but a (1-coherent) \emph{symmetric monoidal} wild category, so let us define this.
\begin{definition}[Monoidal wild categories]
  A monoidal wild category is a wild category $M$ with
  \begin{itemize}
  \item a functor $\otimes : M \times M \to M$,
    \item a unit, i.e.\ an element $I : M$ with natural
      isomorphisms $\lambda_A : I \otimes A \cong A$ and $\rho_A : A
      \otimes I \cong A$,
    \item a family of isomorphisms $\alpha_{A,B,C} : ((A \ox B) \ox C) \cong (A
      \ox (B \ox C))$ natural in all arguments
  such that
  \begin{itemize}
  \item the \textbf{triangle identity} holds, i.e.\ the following diagram commutes,
    \[\begin{tikzcd}[ampersand replacement=\&]
	{(A \otimes I) \otimes B} \&\& {A \otimes (I \otimes B)} \\
	\& {A \otimes B}
	\arrow["{\alpha_{A,I,B}}", from=1-1, to=1-3]
	\arrow["{\rho_A \otimes 1_B}"', from=1-1, to=2-2]
	\arrow["{1_A \otimes \lambda_B}", from=1-3, to=2-2]
    \end{tikzcd}\]
  \item the \textbf{pentagon identity} holds, i.e.\ the following diagram commutes.
\[\begin{tikzcd}[ampersand replacement=\&]
	\& {((A \otimes B) \otimes C) \otimes D} \\
	{(A \otimes (B \otimes C)) \otimes D} \&\& {(A \otimes B) \otimes (C \otimes D)} \\
	{A \otimes ((B \otimes C) \otimes D)} \&\& {A \otimes (B \otimes (C \otimes D))}
	\arrow["{\alpha_{A,B,C} \otimes 1_D}"', from=1-2, to=2-1]
	\arrow["{\alpha_{A,B\otimes C,D}}"', from=2-1, to=3-1]
	\arrow["{\alpha_{A\otimes B,C,D}}", from=1-2, to=2-3]
	\arrow["{\alpha_{A,B,C\otimes D}}", from=2-3, to=3-3]
	\arrow["{1_A\otimes\alpha_{B, C,D}}", from=3-1, to=3-3]
\end{tikzcd}\]
  \end{itemize}
\end{itemize}
\end{definition}
\begin{definition}[1-coherent symmetric monoidal wild categories]
  A 1-coherent symmetric monoidal wild category is a monoidal wild category equipped with a family of isomorphisms $\beta_{A,B} : A \ox B \cong B \ox A$, natural in both arguments, such that
  \begin{itemize}
  \item $\beta_{B,A} \circ \beta_{A,B} = 1_{A\ox B}$,
    \item the \textbf{hexagon identity} holds, i.e.\ the following diagram commutes.
    \[\begin{tikzcd}[ampersand replacement=\&]
	{(A \otimes B) \otimes C} \&\& {(B \otimes A) \otimes C} \\
	{A \otimes (B \otimes C)} \&\& {B\otimes (A\otimes C)} \\
	{(B \otimes C)\otimes A} \&\& {B\otimes (C\otimes A)}
	\arrow["{\beta_{A,B} \otimes 1_C}", from=1-1, to=1-3]
	\arrow["{\alpha_{A,B,C}}"', from=1-1, to=2-1]
	\arrow["{\beta_{A,B\otimes C}}"', from=2-1, to=3-1]
	\arrow["{\alpha_{B,C,A}}"', from=3-1, to=3-3]
	\arrow["{\alpha_{B,A,C}}", from=1-3, to=2-3]
	\arrow["{1_B\otimes \beta_{A,C}}", from=2-3, to=3-3]
\end{tikzcd}\]
  \end{itemize}
\end{definition}

\subsection{Smash products}
The model of the smash product we will use here is given by the
cofibre of the inclusion $A \vee B \hookrightarrow A \times B$,
i.e.\ the following (homotopy) pushout.
\[\begin{tikzcd}[ampersand replacement=\&]
	{A \vee B} \& {A \times B} \\
	1 \& {A \wedge B}
	\arrow[from=1-1, to=1-2]
	\arrow[from=1-1, to=2-1]
	\arrow[from=1-2, to=2-2]
	\arrow[from=2-1, to=2-2]
        \arrow["\lrcorner"{anchor=center, pos=0.125, rotate=180}, draw=none, from=2-2, to=1-1]
\end{tikzcd}\]
For the sake of clarity, let us spell this out in detail by giving explicit names to all constructors of $A \wedge B$. We give the smash product the following self-contained definition.
\begin{definition}[Smash products]\label{def:smash}
  The smash product of two pointed types $A$ and $B$ is the HIT generated by
  \begin{itemize}
  \item a point $\star_\wedge : A \wedge B$,
  \item points $\smashinc{a}{b} : A \wedge B$ for every pair $(a,b) : A \times B$,
  \item paths $\pushl{a}:\smashinc{a}{\star_B} = \star_\wedge$ for every point $a : A$,
  \item paths $\pushr{b} : \smashinc{\star_A}{b} = \star_\wedge$ for every point $b : B$,
  \item a coherence $\pushlr : \pushl{\star_A} = \pushr{\star_B}$.
  \end{itemize}
We always take $A \wedge B$ to be pointed by $\star_\wedge$.
\end{definition}
We remark we could equivalently have defined the smash product by the following pushout.
\[\begin{tikzcd}[ampersand replacement=\&]
	{A + B} \& {A \times B} \\
	{1 + 1} \& {A \wedge B}
	\arrow[from=1-1, to=1-2]
	\arrow[from=1-1, to=2-1]
	\arrow[from=1-2, to=2-2]
	\arrow[from=2-1, to=2-2]
        \arrow["\lrcorner"{anchor=center, pos=0.125, rotate=180}, draw=none, from=2-2, to=1-1]
\end{tikzcd}\]
This definition has the advantage of not having any $2$-dimensional
path constructors but has the disadvantage of having an
additional point constructor. It turns out that Definition~\ref{def:smash} suits our
purposes better, so we will stick with it.

Let us also verify that the smash product is functorial. In what
follows, a pointed function $A \to_\star B$ is a function $f : A \to B$ equipped
with a proof of pointedness $\star_f : f(\star_A) = \star_B$.
\begin{definition}
  For two pointed functions $f : A \to_\star C$ and $g : B \to_\star
  D$, there is an induced map $f \wedge g : A \wedge B \to_\star C
  \wedge D$ defined by
  \begin{align*}
    (f \wedge g)\,(\star_\wedge) &= \star_\wedge\\
    (f \wedge g)\,\langle a,b\rangle &= \langle f(a),g(a)\rangle\\
    \ap{f \wedge g}{\pushl{a}}  &= \ap{\langle f(a) , - \rangle}{\star_g}\cdot \pushl{f(a)}\\
    \ap{f \wedge g}{\pushr{b}}  &= \ap{\langle -,g(b) \rangle}{\star_f}\cdot \pushr{g(b)} \\
    \ap{\mathsf{ap}_{f \wedge g}}{\pushlr}  &= \dots
  \end{align*}
  where the omitted case is a simple coherence which will not matter
  for any future constructions or proofs. We take this map to be
  pointed by \textnormal{$\refl$}.
\end{definition}
We also take the opportunity to mention the commutativity of smash
products. This is trivial since the definition of $A \wedge B$ is entirely
symmetric in both arguments.
\begin{proposition}\label{prop:smash-comm}
  The swap map $A \times B \to B \times A$ induces a pointed
  equivalence $A \wedge B \simeq_\star B \wedge A$.
\end{proposition}

\section{Associativity}
\label{sec:assoc}

Proving that the smash product is associative is far less straightforward
than proving that it is commutative. In fact, even the seemingly direct
task of constructing the associator map is no mean feat. While associativity
has already been verified by \citet{FlorisPhd} and, using a
computer generated proof, by \citet{brunerie18}, let us give a
direct construction of the equivalence. We do this because an explicit description makes the associator easier to trace when verifying e.g.\ the
pentagon identity. For this purpose, let us introduce a new HIT
capturing double smash products in a way which is neutral with respect
to the distribution of parentheses.
\begin{definition}
  Given pointed types $A$, $B$ and $C$, we define the type $\bigwedge(A,B,C)$ to be the HIT generated by
  \begin{itemize}
    \item a point $\star_{2\wedge}$,
    \item points $\langle a,b,c\rangle : \bigwedge(A,B,C)$ for each triple of points $(a,b,c) : A \times B \times C$, \medskip
    \item paths $\pushzero{b,c} : \langle \star_A ,b,c\rangle = \star_{2\wedge}$ for each pair $(b,c) : B \times C$,
    \item paths $\pushone{a,c} : \langle a, \star_B ,c\rangle = \star_{2\wedge}$ for each pair $(a,c) : A \times C$,
    \item paths $\pushtwo{a,b} : \langle a,b,\star_C \rangle = \star_{2\wedge}$ for each pair $(a,b) : A \times B$, \medskip
    \item paths $\pushonetwo{a} : \pushone{a,\star_C} = \pushtwo{a,\star_B} $ for $a : A$,
    \item paths $\pushzerotwo{b} : \pushzero{b,\star_C} = \pushtwo{\star_A,b} $ for $b : B$,
    \item paths $\pushzeroone{c} : \pushzero{\star_B,c}  = \pushone{\star_A,c} $ for $c : C$,
      \medskip
    \item a coherence \textnormal{$\mathsf{push}_{0,1,2}$} filling the the following triangle.
      \[\begin{tikzcd}[ampersand replacement=\&]
	{\pushzero{\star_B,\star_C}} \&\& {\pushtwo{\star_A,\star_B}} \\
	\& {\pushone{\star_A,\star_C}}
	\arrow["{\pushzerotwo{\star_B}}", from=1-1, to=1-3]
	\arrow["{\pushzeroone{\star_C}}"', from=1-1, to=2-2]
	\arrow["{\pushonetwo{\star_A}}"', from=2-2, to=1-3]
\end{tikzcd}\]
  \end{itemize}
\end{definition}
Let us verify that this actually captures a double smash product.
What we need is an equivalence $(A \wedge B) \wedge C \simeq
\bigwedge(A,B,C)$. The underlying map of this equivalence is described
in Table~\ref{tab:alpha} with constructors of $(A \wedge B) \wedge C$
on the left and the corresponding constructors of $\bigwedge(A,B,C)$
on the right. We remark that this correspondence only serves as an
informal sketch of the function---in practice, some simple
coherences are needed for the higher constructors to make it
well-typed.
\begin{table}
\centering
\begin{tabular}{ |c c c| }
  \hline
  $(A \wedge B) \wedge C$ & $\to$ & $\bigwedge(A,B,C)$ \\
  \hline
  $\star_\wedge$ & \squigmapsto & $\star_{2\wedge}$ \\
  $\langle \star_\wedge , c \rangle$ & \squigmapsto & $\star_{2\wedge}$ \\
  $\langle \langle a , b \rangle , c \rangle$ & \squigmapsto & $\langle a,b,c \rangle$ \\
  $\ap{\langle - , c \rangle}{\pushl{a}}$ & \squigmapsto & $\pushone{a,c}$ \\
  $\ap{\langle - , c \rangle}{\pushr{b}}$ & \squigmapsto & $\pushzero{b,c}$ \\
  $\apap{\langle - , c \rangle}{\pushlr}$ & \squigmapsto & $\pushzeroone{c}$ \\
  $\pushl{\star_\wedge}$ & \squigmapsto & $\refl$ \\
  $\pushll{\langle a , b \rangle}$ & \squigmapsto & $\pushtwo{a,b}$ \\
  $\ap{\pushll}{\pushl{a}}$ & \squigmapsto & $\pushonetwo{a}$ \\
  $\ap{\pushll}{\pushr{b}}$ & \squigmapsto & $\pushzerotwo{b}$ \\
  $\apap{\pushll}{\pushlr}$ & \squigmapsto & $\mathsf{push}_{0,1,2}$ \\
  $\pushr{c}$ & \squigmapsto & $\refl$ \\
  $\pushlr$ & \squigmapsto & $\refl$ \\
  \hline
\end{tabular}
\caption{$(A \wedge B) \wedge C$ vs. $\bigwedge(A,B,C)$}
\label{tab:alpha}
\end{table}
Verifying that this map indeed defines an equivalence is somewhat
laborious but direct; the interested reader is referred to the computer formalisation. We get the associativity of the smash product as
a consequence.
\begin{proposition}
  There is a pointed equivalence $\alpha_{A,B,C} : (A \wedge B) \wedge C \simeq_\star A \wedge (B \wedge C)$
\end{proposition}
\begin{proof}
Observe that $\bigwedge(A,B,C)$ is trivially invariant under permutation of the arguments
in the sense that e.g. $\bigwedge(A,B,C) \simeq \bigwedge(B,C,A)$. This allows us to define $\alpha_{A,B,C}$ by the following composition of equivalences.
\begin{align*}
  (A \wedge B) \wedge C \simeq \bigwedge(A,B,C) \simeq \bigwedge(B,C,A) \simeq (B \wedge C) \wedge A \simeq A \wedge (B \wedge C)
\end{align*}
The fact that $\alpha_{A,B,C}$ is pointed holds by $\refl$.
\end{proof}

\section{The heuristic}
\label{sec:heuristic}
When working in HoTT, reasoning about functions defined over iterated smash products quickly
gets out of hand. For instance, in order to verify the pentagon axiom, we
need to reason about functions on the form $((A \wedge B) \wedge C)
\wedge D \to E$. To prove an equality of two such functions $f$ and $g$,
we have to construct, for instance, a dependent path
\begin{align}\label{eq:annoying}
\mathsf{ap}_{\mathsf{ap}_{\mathsf{ap}_{f}\circ \pushll}\circ {\pushll}}{(\pushl{a})}
\longsquiggly
\mathsf{ap}_{\mathsf{ap}_{\mathsf{ap}_{g}\circ \pushll}\circ {\pushll}}{(\pushl{a})}
\end{align}
which boils down to filling a $4$-dimensional cube with highly
non-trivial sides. This is often completely unmanageable in practice.
The best thing we can hope for is that these types of coherence
problems are, in some sense, automatic. This was, in fact, suggested
by~\citet{Brunerie16}, but was never proved nor in any way made formal. In
this section, we will see that this, in fact, is the case.

The first troublesome part of verifying equalities of functions
defined over smash products is the $\pushlr$ constructor. Fortunately,
we do not have to deal with it. Let us denote by $\smashs{A}{B}$ the exact same HIT as $A \wedge B$ but with the $\pushlr$ constructor removed. In other words, it is the following pushout.
\[\begin{tikzcd}[ampersand replacement=\&]
	{A + B} \& {A \times B} \\
	1 \& {\smashs{A}{B}}
	\arrow[from=1-1, to=1-2]
	\arrow[from=1-2, to=2-2]
	\arrow[from=1-1, to=2-1]
	\arrow[from=2-1, to=2-2]
	\arrow["\lrcorner"{anchor=center, pos=0.125, rotate=180}, draw=none, from=2-2, to=1-1]
\end{tikzcd}\]
Let $i$ be the obvious map $\smashs{A}{B}\to A \wedge B$.
\begin{lemma}\label{lem:equaliser}
  For any two maps $f,g : A \wedge B \to C$ satisfying $f\circ i = g \circ i$, we have that $f = g$.
\end{lemma}
\begin{proof}
  The antecedent of the statement provides us with
  \begin{itemize}
  \item a path $p : f(\star_\wedge) = g (\star_\wedge)$,
  \item a homotopy $h : ((a,b) : A \times B) \to f \langle a,b\rangle = g \langle a,b \rangle$,
  \item for each $a : A$, a filler $h_l(a)$ of the square
\[\begin{tikzcd}[ampersand replacement=\&]
	{g\langle a,\star_B\rangle} \&\& {g(\star_\wedge)} \\
	{f\langle a,\star_B\rangle} \&\& {f(\star_\wedge)}
	\arrow["{h(a,\star_B)}", from=2-1, to=1-1]
	\arrow["{\ap{g}{\pushl{a}} }", from=1-1, to=1-3]
	\arrow["p"', from=2-3, to=1-3]
	\arrow["{\ap{f}{\pushl{a}} }"', from=2-1, to=2-3]
\end{tikzcd}\]
\item for each $b : B$, a filler $h_r(b)$ of the square
  \[\begin{tikzcd}[ampersand replacement=\&]
	{g\langle \star_A,b\rangle} \&\& {g(\star_\wedge)} \\
	{f\langle \star_A,b\rangle} \&\& {f(\star_\wedge)}
	\arrow["{h(\star_A,b)}", from=2-1, to=1-1]
	\arrow["{\ap{g}{\pushr{b}} }", from=1-1, to=1-3]
	\arrow["p"', from=2-3, to=1-3]
	\arrow["{\ap{f}{\pushr{b}} }"', from=2-1, to=2-3]
\end{tikzcd}\]
  \end{itemize}
  To prove that $f = g$, we need to provide $p',h',h_l',h_r'$ of the
  same types as above, as well as a filler $h'_{lr}$ of the cube
\[\begin{tikzcd}[ampersand replacement=\&]
	{g\langle \star_A,\star_B\rangle} \&\& {g(\star_\wedge)} \\
	\& {g\langle \star_A,\star_B\rangle} \&\& {g(\star_\wedge)} \\
	{f\langle \star_A,\star_B\rangle} \& {} \& {f(\star_\wedge)} \\
	\& {f\langle \star_A,\star_B\rangle} \&\& {f(\star_\wedge)}
	\arrow[Rightarrow, no head, from=4-2, to=3-1]
	\arrow["{\mathsf{ap}_f(\pushl{\star_A})}"{description}, from=4-2, to=4-4]
	\arrow[Rightarrow, no head, from=4-4, to=3-3]
	\arrow[""{name=0, anchor=center, inner sep=0}, "{\mathsf{ap}_f(\pushr{\star_A})}"{description}, from=3-1, to=3-3]
	\arrow[from=3-3, to=1-3]
	\arrow[from=3-1, to=1-1]
	\arrow[from=4-4, to=2-4]
	\arrow[Rightarrow, no head, from=2-2, to=1-1]
	\arrow["{\mathsf{ap}_g(\pushr{\star_A})}"{description}, from=1-1, to=1-3]
	\arrow[Rightarrow, no head, from=2-4, to=1-3]
	\arrow["{\mathsf{ap}_g(\pushl{\star_A})}"{description}, from=2-2, to=2-4]
	\arrow[shorten <=3pt, from=3-2, to=2-2]
	\arrow[shorten >=5pt, no head, from=4-2, to=0]
\end{tikzcd}\]
where the top and bottom squares are given respectively by
$\mathsf{ap}_{\mathsf{ap}_{g}}(\pushlr)$ and
$\mathsf{ap}_{\mathsf{ap}_{f}}(\pushlr)$, the left- and right-hand side
respectively by $\refl_{h'(\star_A,\star_B)}$ and $\refl_{p'}$ and the front and back respectively by $h_l'(\star_A)$ and $h_r'(\star_B)$.

We set $p' := p$, $h' := h$ and $h_l' := h_l$. For $h_r'$, however, we
need to make an alteration. We construct it as the following composite square
\[\begin{tikzcd}[ampersand replacement=\&]
	{g\langle \star_A,b\rangle} \&\& {g(\star_\wedge)} \& {g(\star_\wedge)} \\
	{f\langle \star_A,b\rangle} \&\& {f(\star_\wedge)} \& {f(\star_\wedge)}
	\arrow["{h(\star_A,b)}", from=2-1, to=1-1]
	\arrow["{\ap{g}{\pushr{b}} }", from=1-1, to=1-3]
	\arrow["p"', from=2-3, to=1-3]
	\arrow["{\ap{f}{\pushr{b}} }"', from=2-1, to=2-3]
	\arrow[Rightarrow, no head, from=1-3, to=1-4]
	\arrow[Rightarrow, no head, from=2-3, to=2-4]
	\arrow["p"', from=2-4, to=1-4]
\end{tikzcd}\]
where the square on the left is $h_r(b)$ and the square on the right is the lid of the cube
\[\begin{tikzcd}[ampersand replacement=\&]
	{g(\star_\wedge)} \&\& {g(\star_\wedge)} \\
	\& {f(\star_\wedge)} \&\& {f(\star_\wedge)} \\
	{g\langle\star_A,\star_B\rangle} \&\& {g\langle\star_A,\star_B\rangle} \\
	\& {f\langle\star_A,\star_B\rangle} \&\& {f\langle\star_A,\star_B\rangle}
	\arrow[Rightarrow, no head, from=4-2, to=4-4]
	\arrow["{h(\star_A,\star_B)}"{description}, from=4-2, to=3-1]
	\arrow[Rightarrow, no head, from=3-1, to=3-3]
	\arrow["{\ap{g}{\pushl{\star_A}}}"{description, pos=0.7}, from=3-1, to=1-1]
	\arrow[Rightarrow, no head, from=1-1, to=1-3]
	\arrow["{\ap{g}{\pushr{\star_B}}}"{description, pos=0.7}, from=3-3, to=1-3]
	\arrow["p"{description}, from=2-2, to=1-1]
	\arrow[Rightarrow, no head, from=2-2, to=2-4]
	\arrow["p"{description}, from=2-4, to=1-3]
	\arrow["{{\ap{f}{\pushl{\star_A}}}}"{description, pos=0.7}, from=4-2, to=2-2]
	\arrow["{\ap{f}{\pushr{\star_B}}}"{description, pos=0.7}, from=4-4, to=2-4]
	\arrow["{h(\star_A,\star_B)}"{description}, from=4-4, to=3-3]
\end{tikzcd}\]
with $h_l(\star_A)$ and $h_r(\star_B)$ as left- and right-hand sides, the action of $f$ and $g$ on $\pushlr$ on the front and
back, and $\refl_{h(\star_A,\star_B)}$ on the bottom. One
can now easily construct the filler $h'_{lr}$ by generalising the
cubes involved and applying path induction.
\end{proof}
Lemma~\ref{lem:equaliser} is useful but does not get us all the way. Complicated
paths like~\eqref{eq:annoying} still need to be constructed,
regardless of what happens with the $\pushlr$ constructor. To strengthen the principle, we will need to introduce the concept of \emph{homogeneous types}.
\begin{definition}
  A pointed type $A$ is homogeneous if for any $a : A$, there is a pointed equivalence
  $(A,\star_A) \simeq_\star (A,a)$.
\end{definition}
The usefulness of homogeneous types is showcased by the following
lemma which was first conjectured for Eilenberg-MacLane spaces in work
leading up to~\cite{BLM22} and later proved and generalised by
Cavallo (and later further generalised by~Buchholtz, Christensen,
G. Taxerås Flaten and Rijke~(\citeyear{taxeras})).
\begin{lemma}[Cavallo]\label{lem:evan}
  Let $f ,g : A \to_\star B$ and let $B$ be homogeneous. If $f = g$ as plain
  functions, then $f = g$ as pointed functions.
\end{lemma}
The same lemma holds for bi-pointed functions $f,g : A \to_\star (B
\to_\star C)$, since the type $(B \to_\star C)$ is homogeneous if $C$
is. This gives a corresponding principle for maps defined over smash
products via the adjunction
$(A \wedge B \to_\star C) \simeq (A \to_\star (B
\to_\star C))$.
\begin{lemma}\label{lem:evan-smash}
  Let $f ,g : A \wedge B \to_\star C$ and let $C$ be homogeneous. If
  $f\langle a,b \rangle = g \langle a , b \rangle$ for all $a : A$ and
  $b : B$, we obtain an equality of pointed functions $f = g$.
\end{lemma}
The following form is useful when dealing with non-pointed functions.
\begin{lemma}\label{lem:evan-smash-unpointed}
  Let $C$ be an arbitrary type and suppose that we have two functions $f ,g : A \wedge B \to C$ with $(C,f(\star_\wedge))$ homogeneous. If
  $f\langle a,b \rangle = g \langle a , b \rangle$ for all $a : A$ and
  $b : B$, then $f = g$.
\end{lemma}
\begin{proof}
  Let $h(a,b) : f\langle a,b \rangle = g \langle a , b \rangle$.
  We know that $g(\star_\wedge) = f(\star_\wedge)$ by the composite path
  \[g(\star_\wedge) \xrightarrow{\ap{g}{\pushl{\star_A}}^{-1}} g\langle \star_A,\star_B \rangle \xrightarrow{h(\star_A,\star_B)^{-1}}{f\langle \star_A,\star_B \rangle} \xrightarrow{\ap{f}{\pushl{\star_A}}} f (\star_\wedge) \]
  Hence, we may view both $f$ and $g$ as pointed functions $A \wedge B \to_\star (C,f(\star_\wedge))$. Now, Lemma~\ref{lem:evan-smash} applies and, in particular, $f = g$ as plain functions. 
\end{proof}
If we could apply Lemma~\ref{lem:evan-smash}~or~\ref{lem:evan-smash-unpointed} when proving the pentagon identity, we would be done immediately. Unfortunately, none of the
types showing up in the statement of the pentagon identity are necessarily homogeneous. There is, however,
some use for the lemmas. Let us first make the following observation: given
two pointed functions $f,g: A \wedge B \to_\star C$ and a homotopy $h :
((a , b) : A \times B) \to f\langle a,b\rangle = g \langle a,b
\rangle$, we can define two functions $L_h : A \to \Omega(C)$ and $R_h
: B \to \Omega(C)$ in terms of $h$ and $\pushll$/$\pushrr$. The
obvious definition of these maps~(which we will spell out in Definition~\ref{def:LR}) will give us a version
of~Lemma~\ref{lem:equaliser} which tells us that if $L_h$ and $R_h$ are
constant, then $f = g$ as plain functions. Now, if either $A$ or $B$ is another smash product, this
\emph{would} be a situation where Lemma~\ref{lem:evan-smash-unpointed} applies since
$\Omega(C)$ (and indeed any path type) is homogeneous. This suggests that Lemma~\ref{lem:evan-smash-unpointed} may be used our advantage when dealing with iterated smash products. Let us try to spell this out more clearly. In fact, we can state this idea without any pointedness assumptions by simply replacing $\Omega(C)$ with the path type $f(\star_\wedge) = g(\star_\wedge)$ which also is homogeneous (and agrees with $\Omega(C)$ whenever $f$ and $g$ are pointed).
\begin{definition}\label{def:LR}
  Let $f,g : A \wedge B \to C$ and let $h : ((a , b) : A \times B) \to f\smashinc{a}{b} = g\smashinc{a}{b}$. This induces, in particular, functions $L_h : A \to f(\star_\wedge) = g(\star_\wedge)$ and $R_h : B \to f(\star_\wedge) = g(\star_\wedge)$ defined by
  \begin{align*}
    L_h(a) &= \ap{f}{\pushl{a}}^{-1} \cdot h(a,\star_B) \cdot \ap{g}{\pushl{a}} \\
    R_h(b) &= \ap{f}{\pushr{b}}^{-1} \cdot h(\star_A,b) \cdot \ap{g}{\pushr{b}} 
  \end{align*}
\end{definition}
We may use $L_h$ and $R_h$ to give a compact induction principle for identities $f = g$. The
following lemma is a direct rewriting of Lemma~\ref{lem:equaliser}.
\begin{lemma}\label{lem:main}
  Let $f ,g: A \wedge B \to C$. The following data yields an equality $f = g$.
  \begin{itemize}
  \item A homotopy $h : ((a , b) : A \times B) \to f\smashinc{a}{b} = g\smashinc{a}{b}$,
  \item equalities $L_h = \mathsf{const}_{L_h(\star_A)}$ and $R_h = \mathsf{const}_{R_h(\star_B)}$.
  \end{itemize}
  If $C$ is a pointed type, $f$ and $g$ are pointed functions and an equality $f = g$ of pointed functions is desired, a further coherence $\star_f = L_h(\star_A) \cdot \star_g$ is required.
\end{lemma}
The key idea now is, we stress again, to use Lemmas~\ref{lem:main}~and~\ref{lem:evan-smash}/\ref{lem:evan-smash-unpointed} iteratively to prove equalities of functions defined over iterated smash products. As a first example, let us consider the problem of proving an equality of functions $f = g$ where $f ,g : (A \wedge B) \wedge C \to D$. The following lemma provides a good estimate of the minimal amount of data needed to construct such an equality.
\begin{lemma}\label{lem:3-eq}
  For two functions $f ,g : (A \wedge B) \wedge C \to D$, the
  following data gives an equality $f = g$.
  \begin{enumerate}[label=(\roman*)]
    \item \label{item:homotopy} A homotopy $h : ((a,b,c) : A \times B \times C) \to f
      \langle a,b,c \rangle = g \langle a,b,c \rangle$,
    \item \label{item:AC-fill} for every pair $(a,c) : A \times C$, a filler of the square
\[\begin{tikzcd}[ampersand replacement=\&]
	{f\langle \star_A,\star_B,c\rangle} \&\& {g\langle \star_A,\star_B,c\rangle} \\
	{f\langle \star_\wedge,c\rangle} \&\& {g\langle \star_\wedge,c\rangle} \\
	{f\langle a,\star_B,c\rangle} \&\& {g\langle a,\star_B,c\rangle}
	\arrow["{h(a,\star_B,c)}"', from=3-1, to=3-3]
	\arrow["{\mathsf{ap}_{f\langle-,c\rangle}(\pushl{a})}", from=3-1, to=2-1]
	\arrow["{\mathsf{ap}_{g\langle-,c\rangle}(\pushl{a})}"', from=3-3, to=2-3]
	\arrow["{\mathsf{ap}_{f\langle-,c\rangle}(\pushl{\star_A})^{-1}}", from=2-1, to=1-1]
	\arrow["{h(\star_A,\star_B,c)}", from=1-1, to=1-3]
	\arrow["{\mathsf{ap}_{g\langle-,c\rangle}(\pushl{\star_A})^{-1}}"', from=2-3, to=1-3]
\end{tikzcd}\]
\item \label{item:BC-fill} for every pair $(b,c) : B \times C$, a filler of the square
\[\begin{tikzcd}[ampersand replacement=\&]
	{f\langle \star_A,\star_B,c \rangle} \&\& {g\langle \star_A,\star_B,c \rangle} \\
	{f\langle \star_\wedge,c\rangle} \&\& {g\langle \star_\wedge,c\rangle} \\
	{f\langle \star_A,b,c\rangle} \&\& {g\langle \star_A,b,c\rangle}
	\arrow["{h(\star_A,b,c)}"', from=3-1, to=3-3]
	\arrow["{\mathsf{ap}_{f\langle-,c\rangle}(\pushr{b})}", from=3-1, to=2-1]
	\arrow["{\mathsf{ap}_{g\langle-,c\rangle}(\pushr{b})}"', from=3-3, to=2-3]
	\arrow["{\mathsf{ap}_{f\langle-,c\rangle}(\pushr{\star_B})^{-1}}", from=2-1, to=1-1]
	\arrow["{\mathsf{ap}_{g\langle-,c\rangle}(\pushr{\star_B})^{-1}}"', from=2-3, to=1-3]
	\arrow["{h(\star_A,\star_B,c)}", from=1-1, to=1-3]
\end{tikzcd}\]
    \item \label{item:AB-fill} for every pair $(a,b) : A \times B$, a filler of the square
\[\begin{tikzcd}[ampersand replacement=\&]
	{f\langle \star_A,\star_B,\star_C\rangle} \&\& {g\langle \star_A,\star_B,\star_C\rangle} \\
	{f\langle \star_\wedge,\star_C \rangle} \&\& {g\langle \star_\wedge,\star_C \rangle} \\
	{f(\star_\wedge)} \&\& {g(\star_\wedge)} \\
	{f\langle a,b,\star_C\rangle} \&\& {g\langle a,b,\star_C\rangle}
	\arrow["{\mathsf{ap}_{f}(\pushll \langle a,b\rangle)}", from=4-1, to=3-1]
	\arrow["{\mathsf{ap}_{f}(\pushl{\star_\wedge}^{-1})}", from=3-1, to=2-1]
	\arrow["{\mathsf{ap}_{f\langle - , \star_C\rangle}(\pushr{\star_B}^{-1}))}", from=2-1, to=1-1]
	\arrow["{h(\star_A,\star_B,\star_C)}", from=1-1, to=1-3]
	\arrow["{\mathsf{ap}_{g\langle - , \star_C\rangle}(\pushr{\star_B}^{-1}))}"', from=2-3, to=1-3]
	\arrow["{\mathsf{ap}_{g}(\pushl{\star_\wedge}^{-1})}"', from=3-3, to=2-3]
	\arrow["{h(a,b,\star_C)}"', from=4-1, to=4-3]
	\arrow["{\mathsf{ap}_{g}(\pushll \langle a,b\rangle)}"', from=4-3, to=3-3]
\end{tikzcd}\]
\item \label{item:C-fill} for every point $c : C$, a filler of the square
\[\begin{tikzcd}[ampersand replacement=\&]
	{f\langle\star_A,\star_B,\star_C\rangle} \&\& {g\langle\star_A,\star_B,\star_C\rangle} \\
	{f\langle\star_\wedge, \star_C\rangle} \&\& {g\langle\star_\wedge, \star_C\rangle} \\
	{f(\star_\wedge)} \&\& {g(\star_\wedge)} \\
	{f\langle\star_\wedge, c\rangle} \&\& {g\langle\star_\wedge, c\rangle} \\
	{f\langle\star_A,\star_B,c\rangle} \&\& {g\langle\star_A,\star_B,c\rangle}
	\arrow["{h(\star_A,\star_B,c)}"', from=5-1, to=5-3]
	\arrow["{\mathsf{ap}_{f\langle -,c\rangle}(\pushr{\star_B})}", from=5-1, to=4-1]
	\arrow["{\mathsf{ap}_{g\langle -,c\rangle}(\pushr{\star_B})}"', from=5-3, to=4-3]
	\arrow["{\mathsf{ap}_f(\pushr{c})}", from=4-1, to=3-1]
	\arrow["{\mathsf{ap}_g(\pushr{c})}"', from=4-3, to=3-3]
	\arrow["{h(\star_A,\star_B,\star_C)}", from=1-1, to=1-3]
	\arrow["{\mathsf{ap}_f(\pushl{\star_\wedge})^{-1}}", from=3-1, to=2-1]
	\arrow["{\mathsf{ap}_{f\langle -,\star_C\rangle}(\pushr{\star_B})^{-1}}", from=2-1, to=1-1]
	\arrow["{\mathsf{ap}_g(\pushl{\star_\wedge})^{-1}}"', from=3-3, to=2-3]
	\arrow["{\mathsf{ap}_{g\langle -,\star_C\rangle}(\pushr{\star_B})^{-1}}"', from=2-3, to=1-3]
\end{tikzcd}\]
  \end{enumerate}
\end{lemma}
\begin{proof}[Proof sketch]
  Suppose we have the given data. We apply Lemma~\ref{lem:main} to $f$ and
  $g$. This breaks the proof up into two $2$ subgoals.
  \begin{itemize}
    \item First, we need to provide a homotopy $k : ((x , c) : (A
      \wedge B) \times C) \to f\langle x,c\rangle = g \langle x,c
      \rangle$. To construct $k$, we fix $c :C$ and note that we may now apply
      Lemma~\ref{lem:main} again: this time to the functions $f(-,c)$
      and $g(-,c)$. This gives us 2 new subgoals.
  \begin{itemize}
    \item First, we need a homotopy $((a , b , c) : A \times B \times
      C) \to f\langle a , b,c \rangle = g \langle a , b , c
      \rangle$. This is given by $h$.
    \item We finally need to show that $L_{h(-,c)}$ and $R_{h(-,c)}$ are
      constant. This boils down to
      providing fillers of the squares which we assumed in
      \ref{item:AC-fill} and \ref{item:BC-fill}.
  \end{itemize}
\item We then need to show that $L_{k}$ and $R_{k}$ are constant. To
  show that $L_k$ is constant, we apply Lemma~\ref{lem:evan-smash-unpointed}. This is justified since the codomain of $L_k$ is homogeneous. Hence, we only need to verify that
  $L_k\langle a,b \rangle = L_k (\star_\wedge)$. Unfolding the
  definitions, we see that this is given by
  assumption~\ref{item:AB-fill}. Note that this is where the explosion of
  complexity would normally happen but, thanks
  to Lemma~\ref{lem:evan-smash-unpointed}, we completely avoid having to verify any
  higher coherences. For $R_k$ we again unfold the
  definitions to see that the path required is provided by~\ref{item:C-fill}. \qedhere
  \end{itemize}
\end{proof}

For completeness, let us state the corresponding lemma for functions
$f ,g : ((A \wedge B) \wedge C) \wedge D \to E$, as these appear in the pentagon identity. The proof is
by~Lemmas~\ref{lem:main},\ref{lem:evan-smash-unpointed}~and~\ref{lem:3-eq}, following the exact same line
of attack as before. We acknowledge that the following result is highly technical but stress that the interesting aspect of it is not the exact details of the statement; rather, we include it to showcase the fact
that only squares are involved, as opposed to the (many) high-dimensional
cubes of coherences which would appear in a naive inductive proof.
\begin{lemma}\label{lem:4-eq}
  For any two functions $f ,g : ((A \wedge B) \wedge C) \wedge D \to E$, the
  following data gives an equality $f = g$:
  \begin{enumerate}[label=(\roman*)]
    \item \label{item:homotopy4} A homotopy $h : ((a,b,c,d) : A \times B \times C \times D) \to f
      \langle a,b,c,d \rangle = g \langle a,b,c,d \rangle$.
    \item \label{item:ABD-fill} For every triple $(a,b,d) : A \times B \times D$, a filler of the square
\[\begin{tikzcd}[ampersand replacement=\&]
	{f\langle \star_A,\star_B,\star_C,d\rangle} \&\& {g\langle \star_A,\star_B,\star_C,d\rangle} \\
	{f\langle \star_\wedge,\star_C , d\rangle} \&\& {g\langle \star_\wedge,\star_C ,d\rangle} \\
	{f\langle \star_\wedge ,d\rangle} \&\& {g\langle \star_\wedge ,d\rangle} \\
	{f\langle a,b,\star_C,d\rangle} \&\& {g\langle a,b,\star_C,d\rangle}
	\arrow["{\mathsf{ap}_{f\langle - , d \rangle}(\pushll \langle a,b\rangle)}", from=4-1, to=3-1]
	\arrow["{\mathsf{ap}_{f\langle - , d \rangle}(\pushl{\star_\wedge}^{-1})}", from=3-1, to=2-1]
	\arrow["{\mathsf{ap}_{f\langle - , \star_C,d\rangle}(\pushr{\star_B}^{-1}))}", from=2-1, to=1-1]
	\arrow["{h(\star_A,\star_B,\star_C)}", from=1-1, to=1-3]
	\arrow["{\mathsf{ap}_{g\langle - , \star_C,d\rangle}(\pushr{\star_B}^{-1}))}"', from=2-3, to=1-3]
	\arrow["{\mathsf{ap}_{g\langle - , d \rangle}(\pushl{\star_\wedge}^{-1})}"', from=3-3, to=2-3]
	\arrow["{h(a,b,\star_C)}"', from=4-1, to=4-3]
	\arrow["{\mathsf{ap}_{g\langle - , d \rangle}(\pushll \langle a,b\rangle)}"', from=4-3, to=3-3]
\end{tikzcd}\]
\item \label{item:ACD-fill} For every triple $(a,c,d) : A \times C\times D$, a filler of the square
\[\begin{tikzcd}[ampersand replacement=\&]
	{f\langle \star_A,\star_B,c,d\rangle} \&\& {g\langle \star_A,\star_B,c,d\rangle} \\
	{f\langle \star_\wedge,c,d\rangle} \&\& {g\langle \star_\wedge,c,d\rangle} \\
	{f\langle a,\star_B,c,d\rangle} \&\& {g\langle a,\star_B,c,d\rangle}
	\arrow["{h(a,\star_B,c,d)}"', from=3-1, to=3-3]
	\arrow["{\mathsf{ap}_{f\langle-,c,d\rangle}(\pushl{a})}", from=3-1, to=2-1]
	\arrow["{\mathsf{ap}_{g\langle-,c,d\rangle}(\pushl{a})}"', from=3-3, to=2-3]
	\arrow["{\mathsf{ap}_{f\langle-,c,d\rangle}(\pushl{\star_A})^{-1}}", from=2-1, to=1-1]
	\arrow["{h(\star_A,\star_B,c,d)}", from=1-1, to=1-3]
	\arrow["{\mathsf{ap}_{g\langle-,c,d\rangle}(\pushl{\star_A})^{-1}}"', from=2-3, to=1-3]
\end{tikzcd}\]
\item \label{item:BCD-fill} For every triple $(b,c,d) : B \times C \times D$, a filler of the square
\[\begin{tikzcd}[ampersand replacement=\&]
	{f\langle \star_A,\star_B,c,d \rangle} \&\& {g\langle \star_A,\star_B,c,d \rangle} \\
	{f\langle \star_\wedge,c,d\rangle} \&\& {g\langle \star_\wedge,c,d\rangle} \\
	{f\langle \star_A,b,c,d\rangle} \&\& {g\langle \star_A,b,c,d\rangle}
	\arrow["{h(\star_A,b,c,d)}"', from=3-1, to=3-3]
	\arrow["{\mathsf{ap}_{f\langle-,c,d\rangle}(\pushr{b})}", from=3-1, to=2-1]
	\arrow["{\mathsf{ap}_{g\langle-,c,d\rangle}(\pushr{b})}"', from=3-3, to=2-3]
	\arrow["{\mathsf{ap}_{f\langle-,c,d\rangle}(\pushr{\star_B})^{-1}}", from=2-1, to=1-1]
	\arrow["{\mathsf{ap}_{g\langle-,c,d\rangle}(\pushr{\star_B})^{-1}}"', from=2-3, to=1-3]
	\arrow["{h(\star_A,\star_B,c,d)}", from=1-1, to=1-3]
\end{tikzcd}\]
\item \label{item:CD-fill} For every pair $(c,d) : C \times D$, a filler of the square
\[\begin{tikzcd}[ampersand replacement=\&, row sep = 1.5em]
	{f\langle\star_A,\star_B,\star_C,d\rangle} \&\& {g\langle\star_A,\star_B,\star_C\rangle} \\
	{f\langle\star_\wedge, \star_C,d\rangle} \&\& {g\langle\star_\wedge, \star_C\rangle} \\
	{f\langle\star_\wedge,d\rangle} \&\& {g\langle\star_\wedge,d\rangle} \\
	{f\langle\star_\wedge, c,d\rangle} \&\& {g\langle\star_\wedge, c,d\rangle} \\
	{f\langle\star_A,\star_B,c,d\rangle} \&\& {g\langle\star_A,\star_B,c,d\rangle}
	\arrow["{h(\star_A,\star_B,c,d)}"', from=5-1, to=5-3]
	\arrow["{\mathsf{ap}_{f\langle -,c,d\rangle}(\pushr{\star_B})}", from=5-1, to=4-1]
	\arrow["{\mathsf{ap}_{g\langle -,c,d\rangle}(\pushr{\star_B})}"', from=5-3, to=4-3]
	\arrow["{\mathsf{ap}_{f\langle - , d \rangle}(\pushr{c})}", from=4-1, to=3-1]
	\arrow["{\mathsf{ap}_{g\langle - , d \rangle}(\pushr{c})}"', from=4-3, to=3-3]
	\arrow["{h(\star_A,\star_B,\star_C,d)}", from=1-1, to=1-3]
	\arrow["{\mathsf{ap}_{f\langle - , d \rangle}(\pushl{\star_\wedge})^{-1}}", from=3-1, to=2-1]
	\arrow["{\mathsf{ap}_{f\langle -,\star_C,d\rangle}(\pushr{\star_B})^{-1}}", from=2-1, to=1-1]
	\arrow["{\mathsf{ap}_{g\langle - , d \rangle}(\pushl{\star_\wedge})^{-1}}"', from=3-3, to=2-3]
	\arrow["{\mathsf{ap}_{g\langle -,\star_C,d\rangle}(\pushr{\star_B})^{-1}}"', from=2-3, to=1-3]
\end{tikzcd}\]
\item For every $d : D$, a filler of the square
\[\begin{tikzcd}[ampersand replacement=\& , row sep = 1.5em]
	{f\langle\star_A,\star_B,\star_C,\star_D\rangle} \&\& {g\langle\star_A,\star_B,\star_C,\star_D\rangle} \\
	{f\langle\star_\wedge,\star_C,\star_D\rangle} \&\& {g\langle\star_\wedge,\star_C,\star_D\rangle} \\
	{f\langle\star_\wedge,\star_D\rangle} \&\& {g\langle\star_\wedge,\star_D\rangle} \\
	{f(\star_\wedge)} \&\& {g(\star_\wedge)} \\
	{f\langle\star_\wedge,d\rangle} \&\& {g\langle\star_\wedge,d\rangle} \\
	{f\langle\star_\wedge,\star_C,d\rangle} \&\& {g\langle\star_\wedge,\star_C,d\rangle} \\
	{f\langle\star_A,\star_B,\star_C,d\rangle} \&\& {g\langle\star_A,\star_B,\star_C,d\rangle}
	\arrow["{h(\star_A,\star_B,\star_C,d)}"', from=7-1, to=7-3]
	\arrow["{\mathsf{ap}_{f\langle-,\star_C,d\rangle}(\pushl{\star_A}))}", from=7-1, to=6-1]
	\arrow["{\mathsf{ap}_{g\langle-,\star_C,d\rangle}(\pushl{\star_A}))}"', from=7-3, to=6-3]
	\arrow["{\mathsf{ap}_{f\langle-,d\rangle}(\pushl{\star_\wedge}))}", from=6-1, to=5-1]
	\arrow["{\mathsf{ap}_f(\pushr{d})}", from=5-1, to=4-1]
	\arrow["{\mathsf{ap}_{f}(\pushl{\star_\wedge}))^{-1}}", from=4-1, to=3-1]
	\arrow["{\mathsf{ap}_{f\langle-,\star_D\rangle}(\pushl{\star_\wedge}))^{-1}}", from=3-1, to=2-1]
	\arrow["{\mathsf{ap}_{f\langle-,\star_C,\star_D\rangle}(\pushl{\star_A}))^{-1}}", from=2-1, to=1-1]
	\arrow["{h(\star_A,\star_B,\star_C,\star_D)}", from=1-1, to=1-3]
	\arrow["{\mathsf{ap}_{g\langle-,d\rangle}(\pushl{\star_\wedge}))}"', from=6-3, to=5-3]
	\arrow["{\mathsf{ap}_g(\pushr{d})}"', from=5-3, to=4-3]
	\arrow["{\mathsf{ap}_{g}(\pushl{\star_\wedge}))^{-1}}"', from=4-3, to=3-3]
	\arrow["{\mathsf{ap}_{g\langle-,\star_C,\star_D\rangle}(\pushl{\star_A}))^{-1}}"', from=2-3, to=1-3]
	\arrow["{\mathsf{ap}_{g\langle-,\star_D\rangle}(\pushl{\star_\wedge}))^{-1}}"', from=3-3, to=2-3]
\end{tikzcd}\]
\item \label{item:ABC-fill} For every triple $(a,b,c) : A \times B \times C$, a filler of the square
  \[\begin{tikzcd}[ampersand replacement=\&, row sep = 1.5em]
	{f\langle \star_A,\star_B,\star_C\star_D\rangle} \&\& {g\langle \star_A,\star_B,\star_C\star_D\rangle} \\
	{f\langle \star_\wedge,\star_C,\star_D\rangle} \&\& {g\langle \star_\wedge,\star_C,\star_D\rangle} \\
	{f\langle \star_\wedge,\star_D\rangle} \&\& {g\langle \star_\wedge,\star_D\rangle} \\
	{f(\star_\wedge)} \&\& {g(\star_\wedge)} \\
	{f\langle a,b,c,\star_D\rangle} \&\& {g\langle a,b,c,\star_D\rangle}
	\arrow["{h(a,b,c,\star_D)}"', from=5-1, to=5-3]
	\arrow["{\mathsf{ap}_f(\pushll\langle a,b,c\rangle)}", from=5-1, to=4-1]
	\arrow["{\mathsf{ap}_g(\pushll\langle a,b,c\rangle)}"', from=5-3, to=4-3]
	\arrow["{\mathsf{ap}_g(\pushl{\star_\wedge}))^{-1}}"', from=4-3, to=3-3]
	\arrow["{\mathsf{ap}_f(\pushl{\star_\wedge}))^{-1}}", from=4-1, to=3-1]
	\arrow["{\mathsf{ap}_{f\langle -,\star_D\rangle}(\pushl{\star_\wedge}))^{-1}}", from=3-1, to=2-1]
	\arrow["{\mathsf{ap}_{f\langle -,\star_C,\star_D\rangle}(\pushl{\star_A}))^{-1}}", from=2-1, to=1-1]
	\arrow["{h(\star_A,\star_B,\star_C\star_D)}", from=1-1, to=1-3]
	\arrow["{\mathsf{ap}_{g\langle -,\star_D\rangle}(\pushl{\star_\wedge}))^{-1}}"', from=3-3, to=2-3]
	\arrow["{\mathsf{ap}_{g\langle -,\star_C,\star_D\rangle}(\pushl{\star_A}))^{-1}}"', from=2-3, to=1-3]
  \end{tikzcd}\]
  \end{enumerate}
\end{lemma}
Let us make three observations about~Lemmas~\ref{lem:3-eq}~and~\ref{lem:4-eq}.
\begin{enumerate}
  \item In both statements, the different pieces of data are almost
    completely mutually independent. The only meaningful choice we can make
    is that of the homotopy $h$. This means that we are free
    to provide any proofs we like for the remaining steps without
    having to worry about future coherences.
  \item In many concrete cases (especially those relating to the symmetric
    monoidal structure of the smash product), the homotopy $h$ will be
    defined by $h(a_1,\dots,a_n) := \refl$. This means that all other
    data we need to provide are equalities of composite paths defined entirely in terms of applications
    of $f$ and $g$ on $\pushll$ and $\pushrr$---something we can usually simply
    unfold to something (hopefully) simple.
  \item Going from Lemma~\ref{lem:3-eq} to Lemma~\ref{lem:4-eq}, we see that
    only two additional assumptions are needed. If we were to
    increase the number of copies of smash products appearing in the
    domain by one, we would only need to provide two additional
    squares (and still no higher coherences). Hence, the complexity of
    such proofs grows linearly with the complexity of the domain. For
    comparison, if we were to resort to a naive proof by a deep smash
    product induction, the amount of data needed would grow
    exponentially.
\end{enumerate}
While it is possible to
generalise~Lemmas~\ref{lem:3-eq}~and~\ref{lem:4-eq}~to a statement
concerning $n$-fold smash products, let us, for now, only summarise the fundamental idea behind Lemmas~\ref{lem:3-eq}~and~\ref{lem:4-eq} in terms of an informal heuristic.
\begin{heuristic}\label{heur}
  \it
  To show that two functions $f,g : \bigwedge_{i \leq n} A_i \to B$ are equal, it suffices, by iterative application Lemmas~\ref{lem:main},~\ref{lem:evan-smash-unpointed}~and~\ref{lem:evan}, to provide a family of paths $h(x_1,\dots,x_n) : f \langle x_1,\dots,x_n \rangle = g \langle x_1 , \dots,x_n\rangle$ for $x_i : A_i$ and to show that it is coherent with $f$ and $g$ on any \textbf{single} application of $\pushll$ or $\pushrr$ (e.g. $\ap{\langle -,x_{i+1},\dots,x_n\rangle}{\pushll \langle x_1,\dots,x_{i-1} \rangle}$. Furthermore, if an equality of pointed functions is required, we need to provide a filler of the following square:
\[\begin{tikzcd}[ampersand replacement=\&]
	{f(\star_\wedge)} \& {\star_B} \& {g(\star_\wedge)} \\
	\vdots \&\& \vdots \\
	{f\langle\star_\wedge,\star_{A_n}\rangle} \&\& {g\langle\star_\wedge,\star_{A_n}\rangle} \\
	{f\langle\star_{A_1},\dots, \star_{A_{n}}\rangle} \&\& {g\langle\star_{A_1},\dots, \star_{A_{n}}\rangle}
	\arrow[from=4-1, to=3-1]
	\arrow[from=3-1, to=2-1]
	\arrow["{\mathsf{ap}_f{(\pushr{\star_A})}}", from=2-1, to=1-1]
	\arrow["{\mathsf{ap}_g{(\pushr{\star_A})}}"',from=2-3, to=1-3]
	\arrow[from=3-3, to=2-3]
	\arrow[from=4-3, to=3-3]
	\arrow["{h(\star_{A_1},\dots,\star_{A_n})}"', from=4-1, to=4-3]
	\arrow["{\star_f}", from=1-1, to=1-2]
	\arrow["{\star_g^{-1}}", from=1-2, to=1-3]
\end{tikzcd}\]
\end{heuristic}
The reader who is looking for a precise mathematical statement of the above is referred to Section~\ref{sec:equaliser} where we also introduce some additional machinery required to make the idea formal. For now, we will settle for this informal heuristic---in practice, when working with functions over a small fixed number of smash products, we do not quite need the full strength of a general theorem.

\section{The symmetric monoidal structure}
\label{sec:symmetric}
Let us reap the fruits of our labour and show that the smash product defines a (1-coherent)
symmetric monoidal product on the universe of pointed types. We will
not verify all axioms here, since this is neither very instructive nor
very interesting. Instead, we sketch the proofs of the two most technical
properties and refer the interested reader to the computer formalisation.

\begin{proposition}\label{prop:hexagon}
The smash product satisfies the hexagon identity, i.e. we have an
equality of pointed functions $H_0 = H_1$ where $H_0$ and $H_1$ are
defined as the composites of each side of the following hexagon.
\[\begin{tikzcd}[ampersand replacement=\&]
	{(A\wedge B)\wedge C} \&\& {(B\wedge A) \wedge C} \\
	{A \wedge(B\wedge C)} \&\& {B\wedge (A \wedge C)} \\
	{(B\wedge C) \wedge A} \&\& {B\wedge(C\wedge A)}
	\arrow["{\alpha_{A,B,C}}"', from=1-1, to=2-1]
	\arrow["{\beta_{A,B}\wedge 1_C}", from=1-1, to=1-3]
	\arrow["{\alpha_{B,A,C}}", from=1-3, to=2-3]
	\arrow["{\beta_{A,B\wedge C}}"', from=2-1, to=3-1]
	\arrow["{\alpha_{B,C,A}}"', from=3-1, to=3-3]
	\arrow["{1_B\wedge\beta_{A,C}}", from=2-3, to=3-3]
	\arrow["{H_0}"{description}, bend right=15 , dashed, from=1-1, to=3-3]
	\arrow["{H_1}"{description}, bend left=15, dashed, from=1-1, to=3-3]
\end{tikzcd}\]
\end{proposition}
\begin{proof}[Proof sketch]
  We show the statement by an application of our heuristic which, in this case, takes the form of~Lemma~\ref{lem:3-eq}. We provide the data as follows:
  \begin{enumerate}
    \item For the homotopy $h(a,b,c) : H_0\langle a,b,c\rangle = H_1\langle a,b,c \rangle$
    we simply choose $h(a,b,c) := \refl$, since both sides compute to $\langle
    b,c,a\rangle$.
  \item E.g. the fourth datum in~Lemma~\ref{lem:3-eq} computes to\footnote{By
    `computes to' we do not mean `normalises in \Agda to'. We mean
    `compute' in the manual sense, i.e. by tracing $H_0$ and $H_1$ on
    the point and path constructors involved. Direct normalisation in
    \Agda produces rather large and unmanageable terms. However, using
    \Agda to normalise the terms in a more controlled manner
    (i.e. step-by-step) is very useful, as a sanity check, for inspecting the action of
    $H_0$ and $H_1$ on the path constructors involved. } the following square
    filling problem:
\[\begin{tikzcd}[ampersand replacement=\&]
	{\langle \star_B,\star_C,\star_A\rangle} \& {\langle \star_B,\star_C,\star_A\rangle} \\
	{\wedge_C} \& {\wedge_C} \\
	{\wedge_C} \& {\wedge_C} \\
	{\langle b,\star_C,a\rangle} \& {\langle b,\star_C,a\rangle}
	\arrow["{\mathsf{ap}_{\langle b,-\rangle}(\pushr{a}) \cdot \pushl{b} }", from=4-1, to=3-1]
	\arrow["\refl"', from=4-1, to=4-2]
	\arrow["\refl", from=3-1, to=2-1]
	\arrow["{\mathsf{ap}_{\langle b,-\rangle}(\pushr{a}) \cdot \pushl{b} }"', from=4-2, to=3-2]
	\arrow["\refl"', from=3-2, to=2-2]
	\arrow["{\pushl{\star_B}^{-1} \cdot \mathsf{ap}_{\langle \star_B,-\rangle}(\pushl{\star_C})^{-1}}", from=2-1, to=1-1]
	\arrow["{\pushl{\star_B}^{-1} \cdot \mathsf{ap}_{\langle \star_B,-\rangle}(\pushl{\star_C})^{-1}}"', from=2-2, to=1-2]
	\arrow["\refl", from=1-1, to=1-2]
\end{tikzcd}\]
which is solved by $\refl$.
\item The remaining squares are computed and solved in an identical manner.
\item For the pointedness, we need to fill the square outlined in end of the statement of the~\hyperref[heur]{heuristic}. This is equally direct since $\star_{H_0} = \star_{H_1} = \refl$, which holds because all functions involved in the definitions of $H_0$ and $H_1$ are pointed by $\refl$. 
  \end{enumerate}
\end{proof}

\begin{proposition}\label{prop:pentagon}
The pentagon identity holds for the smash product, i.e. we have an
equality of pointed functions $P_0 = P_1$ where $P_0$ and $P_1$ are defined as the composites of each side of the following pentagon.
\[\begin{tikzcd}[ampersand replacement=\&]
	\&\& {((A \wedge B) \wedge C) \wedge D} \\
	{(A \wedge (B \wedge C)) \wedge D} \&\&\&\& {(A \wedge B) \wedge (C \wedge D)} \\
	{A \wedge ((B \wedge C) \wedge D)} \&\&\&\& {A \wedge (B\wedge (C\wedge D))}
	\arrow["{\alpha_{A,B,C} \wedge 1_D}"', from=1-3, to=2-1]
	\arrow["{\alpha_{A,B\wedge C,D}}"', from=2-1, to=3-1]
	\arrow["{\alpha_{A\wedge B,C,D}}", from=1-3, to=2-5]
	\arrow["{\alpha_{A,B,C\wedge D}}", from=2-5, to=3-5]
	\arrow["{1_A \wedge \alpha_{B,C,D}}"', from=3-1, to=3-5]
	\arrow["{P_0}"{description}, shift left=2, bend right=30, shorten <=7pt, shorten >=9pt , dashed, from=1-3, to=3-5]
	\arrow["{P_1}"{description}, bend right=5, dashed, from=1-3, to=3-5]
\end{tikzcd}\]
\end{proposition}
\begin{proof}
  The statement follows easily by the heuristic, which in this case
  corresponds to~Lemma~\ref{lem:4-eq}. The proof is identical to the proof
  of~Proposition~\ref{prop:hexagon} and follows simply by evaluating $P_0$ and
  $P_1$ on the $1$-dimensional path constructors involved and noting
  that all square-filling problems listed in~Lemma~\ref{lem:4-eq} become
  trivial.
\end{proof}
All other axioms defining a 1-coherent symmetric monoidal wild category follow in
the same direct manner and we, after some rather mechanical labour (see computer formalisation),
easily arrive at the main result.
\begin{therm}\label{thm:main}
  The universe of pointed types forms a 1-coherent symmetric monoidal wild category
  with the smash product as tensor product.
\end{therm}

\subsection{Back to Brunerie's thesis}
A first consequence of the fact that Theorem~\ref{thm:main} finally has a complete and
computer formalised proof is the correctness of
Brunerie's PhD thesis~(\citeyear{Brunerie16}). While a computer
formalisation of the main results of the thesis was presented
by~\cite{LM23}, this formalisation did not stay completely true to
Brunerie's original proof. Indeed, certain proofs and constructions
were reworked in order not to rely on results concerning smash
products. Most notably, the \emph{cup product} was redefined using an
alternative definition by~\citet[Section 4.1]{BLM22}.
In Brunerie's thesis, the cup product (on integral Eilenberg-MacLane spaces) is defined as a
map $\smile : K(\mathbb{Z},n)\wedge K(\mathbb{Z},m) \to {K(\mathbb{Z},n+m)}$
where, for $n ,m \geq 1$, we have $K(\mathbb{Z},n) := \truncT{S^n}_{n}$, i.e.\ the $n$-truncation of the $n$-sphere. Here, we define $S^n := \Sigma^{n+1}(\emptyset)$, i.e.\ it is the $(n+1)$-fold suspension of the empty type, where, recall, the suspension of a type $A$, is the HIT $\Sigma(A)$ generated by
\begin{itemize}
\item two points $\north,\south : \Sigma(A)$,
  \item paths $\merid{(a)} : \north = \south$ for each $a:A$.
\end{itemize}
Brunerie's construction of the cup product is by means of a lift from the corresponding map on spheres.
\[
\begin{tikzcd}[ampersand replacement=\&]
	{K(\mathbb{Z},n)\wedge K(\mathbb{Z},m)} \&\& {K(\mathbb{Z},n+m)} \\
	{S^n \wedge S^m} \&\& {S^{n+m}}
	\arrow["{\trunc{-}\,\wedge \,\trunc{-}}", from=2-1, to=1-1]
	\arrow["{\wedge_{n,m}}"', from=2-1, to=2-3]
	\arrow["\trunc{-}"', from=2-3, to=1-3]
	\arrow["\smile", dashed, from=1-1, to=1-3]
\end{tikzcd}
\]
Above, the map $\wedge_{n,m} : S^n \wedge S^m \to S^{n+m}$ is the
canonical equivalence defined in e.g.~\citet[Proposition 4.2.2]{Brunerie16}. With this construction, most properties of the cup product can be shown by showing that the corresponding properties hold for $\wedge_{n,m}$. In particular, graded-commutativity and associativity follow, respectively, from the following two propositions.
\begin{proposition}[Brunerie, Proposition 4.2.4]\label{prop:cup-comm}
  The following diagram commutes.
  \[
\begin{tikzcd}[ampersand replacement=\&]
	{S^n \wedge S^m} \&\& {S^m \wedge S^n} \\
	{S^{n+m}} \&\& {S^{n+m}}
	\arrow[from=1-1, to=1-3]
	\arrow["{\wedge_{n,m}}"', from=1-1, to=2-1]
	\arrow["{(-1)^{nm}}"', from=2-1, to=2-3]
	\arrow["{\wedge_{m,n}}", from=1-3, to=2-3]
\end{tikzcd}
\]
where $(-1) : S^n \to S^n$ is defined by sending $\merid{(a)}$ to $\merid{(a)}^{-1}$.
\end{proposition}
\begin{proposition}[Brunerie, Proposition 4.2.2]\label{prop:cup-assoc}
  The following diagram commutes.
\[
\begin{tikzcd}[ampersand replacement=\& , row sep = small]
	{(S^n\wedge S^m) \wedge S^k} \&\& {S^{n+m}\wedge S^k} \\
	\&\&\&\& {S^{n+n+k}} \\
	{S^n\wedge(S^m \wedge S^k)} \&\& {S^{n}\wedge S^{m+k}}
	\arrow["{\alpha_{S^n,S^m,S^k}}"', from=1-1, to=3-1]
	\arrow["{(\wedge_{n,m}) \wedge 1_{S^k}}", from=1-1, to=1-3]
	\arrow["{ 1_{S^n} \wedge (\wedge_{m,k})}"', from=3-1, to=3-3]
	\arrow["{\wedge_{n+m,k}}", from=1-3, to=2-5]
	\arrow["{\wedge_{n,m+k}}"', from=3-3, to=2-5]
\end{tikzcd}
\]
\end{proposition}
These results are proved by Brunerie using proofs which rely on the symmetric monoidal structure of the smash product. While Proposition~\ref{prop:cup-comm} should be provable using the partial proof of symmetric monoidality due to~\citet[Section 4.3]{FlorisPhd}, Brunerie's proof of Proposition~\ref{prop:cup-assoc} relies on the pentagon identity which, until now, has been open. Thus, with Theorem~\ref{thm:main} we can finally conclude the following.
\begin{therm}\label{thm:brunerie-repair}
  The cup product, as constructed in \citet[Definition 5.1.6]{Brunerie16}, forms a graded-commutative and associative multiplication $K(\mathbb{Z},n) \wedge K(\mathbb{Z},m) \to K(\mathbb{Z},n+m)$.
\end{therm}
We remark that we have not provided a computer formalisation of the above since there already exist well developed computer formalisations of the cup product and cohomology rings~\citep{BLM22,LLM23,LM24}. 


\section{A formal statement of the heuristic}
\label{sec:equaliser}
In a classical setting, the iterated smash product $\bigwedge_{i \leq n} {A_i}$ can be described as the quotient space $(A_0 \times \dots \times A_{n}) / \fw{i\leq n}{A_i}$ where $$\fw{i\leq n}{A_i} := \{(a_0,\dots,a_{n}) \mid a_i = \star_{A_i} \text{ for some $i$} \}\subset A_0 \times \dots \times A_{n}$$
It is unclear how to mimic this construction in HoTT in a useful way, as the naive definition forces us to add, in a systematic way, an exponential number of higher coherences. However, the naive definition may still prove useful. Here, by `naive definition', we mean the disjoint union
\begin{align*}
  \fwred{i \leq n}{A_i} := \bigsqcup_{i \leq n}\left(A^{(i)}_{0} \times \dots \times A^{(i)}_{n} \right)\hspace{1cm} &\mathsf{ where }\hspace{.2cm}A^{(i)}_j = \begin{cases} A_j & \text{ if $i \neq j$}\\
    1 & \text{ if $i = j$} 
    \end{cases}
\end{align*}
While the cofibre of the canonical map $\fwred{i \leq n}{A_i} \to A_0 \times \dots \times A_{n}$ does not quite give us $\bigwedge_{i \leq n} {A_i}$, it \emph{does}, however, provide a useful approximation.
This is, in fact, precisely the content of our heuristic. We will soon make this precise. Before this, let us give a less set-theoretic definition of $\fwred{i \leq n}{A_i}$ which is easier to work with in HoTT.
\begin{definition}
  Given pointed types $A_0,\dots,A_{n}$ we define $\fwred{i\leq n}{A_i}$ by induction on $n$ as follows.
  \[
  \fwred{i\leq n}{A_i} :=
  \begin{cases}
    1 & \text{ if $n = 0$}\\
    (\fwred{i\leq {n-1}}{A_i} \times A_{n}) + (A_0\times \dots \times {A_{n-1}}) & \text{ if $n > 0$}
    \end{cases}
  \]
\end{definition}
Let us, for clarity, explicitly describe the canonical map $\gamma_n : \fwred{i\leq n}{A_i}  \to A_0 \times \dots \times A_n$. It is recursively defined with $\gamma_0 : 1 \to A_0$ picking out the basepoint $\star_{A_{0}}$ and \[ \gamma_{n} : (\fwred{i\leq {n-1}}{A_i} \times A_{n}) + (A_0\times \dots \times {A_{n-1}})  \to A_0\times \dots \times A_n \] being defined by $\gamma_n = \gamma_n^{(1)} + \gamma_{n}^{(2)}$
where
\begin{align*}
  \gamma_n^{(1)}(x,a_{n}) &= (\gamma_{n-1}(x),a_n)\\
  \gamma_n^{(2)}(a_0,\dots,a_{n-1}) &= (a_0,\dots,a_{n-1},\star{A_{n}})
\end{align*}
\begin{definition}
  Given pointed types $A_0,\dots,A_n$, we define $\widetilde{\bigwedge}_{i \leq n}{A_i}$ to be the cofibre of \linebreak $\gamma_n : \fwred{i\leq n}{A_i}  \to A_0 \times \dots \times A_n$, i.e. the following pushout.
  \[
\begin{tikzcd}[ampersand replacement=\&]
	{\fwred{i\leq n}{A_i} } \&\& {A_0 \times \dots \times A_n} \\
	{1} \&\& {\widetilde{\bigwedge}_{i \leq n}{A_i}}
	\arrow["{\gamma_n}", from=1-1, to=1-3]
	\arrow[from=1-1, to=2-1]
	\arrow[from=2-1, to=2-3]
	\arrow["\lrcorner"{anchor=center, pos=0.125, rotate=180}, draw=none, from=2-3, to=1-1]
	\arrow[from=1-3, to=2-3]
\end{tikzcd}
\]
We give its constructors explicit names and write
\begin{itemize}
\item $\star_{\widetilde{\wedge}}$ for its basepoint,
\item $\langle a_0,\dots,a_n \rangle$ for its underlying points,
\item $\pushop^{(1)}(x) : \langle \gamma_n^{(1)}(x) \rangle = \star_{\widetilde{\wedge}}$ for the first coherence given by the pushout square,
\item $\pushop^{(2)}(x) : \langle \gamma_n^{(2)}(x) \rangle = \star_{\widetilde{\wedge}}$ for the second coherence given by the pushout square.
\end{itemize}
\end{definition}

It is easy to see that the composition $\fwred{i \leq n}{A_i} \to A_0\times \dots \times A_n \to \bigwedge_{i \leq n} {A_i}$ is null-homotopic. Hence, there is a basepoint preserving inclusion (of constructors) $\iota : \widetilde{\bigwedge}_{i \leq n}{A_i} \to \bigwedge_{i \leq n}{A_i} $ s.t. $\iota\langle a_0,\dots,a_n\rangle = \langle a_0,\dots,a_n\rangle$.

We also remark that there is an inclusion $\uparrow: \left(\widetilde{\bigwedge}_{i \leq {n-1}} {A_i}\right)\times A_{n} \to \widetilde{\bigwedge}_{i \leq {n}} {A_i}$. It is defined by
\begin{align*}
  \uparrow(\star_{\widetilde{\wedge}},a_n) &:= \star_{\widetilde{\wedge}}\\
  \uparrow(\langle a_0,\dots,a_{n-1} \rangle,a_n) &:= \langle a_0,\dots,a_n \rangle\\
  \ap{\uparrow(-,a_n)}{\push{x}} &:= \pushop^{(1)}{(x,a_n)}
\end{align*}
Furthermore, $\uparrow$ commutes with $\iota$ in the sense that, for $x : \widetilde{\bigwedge}_{i \leq {n-1}} {A_i}$ and $a_n : A_{n}$, we have the following path (in $\bigwedge_{i \leq n}{A_i}$): $$\uparrow^{\mathsf{coh}}(x , a_n) \,: \iota\left(\uparrow(x,a_n)\right) = \langle \iota(x) , a_n \rangle$$
One constructs $\uparrow^{\mathsf{coh}}(x , a_n)$ very directly by induction on $x$. Let us give the construction when $x$ is a point constructor. When $x:=\langle a_0,\dots,a_{n-1} \rangle$, the statement holds by $\refl$. When $x:=\star_{\widetilde{\wedge}}$, the type of $\uparrow^{\mathsf{coh}}(x , a_n)$ reduces to $\star_{\wedge} = \langle \star_{\wedge},a_n \rangle$ which holds by $\pushr{a_n}^{-1}$. Hence, we define
\begin{align}
  \label{eq:cohdef-pairs}
  \uparrow^{\mathsf{coh}}(x , a_n):=\refl\\
  \label{eq:cohdef}
  \uparrow^{\mathsf{coh}}(\star_{\widetilde{\wedge}} , a_n) := \pushr{a_n}^{-1}
\end{align}
The higher coherence is very straightforward and we omit it for the sake of readability. We are now ready for the key result of the section. It makes precise to what extent $\widetilde{\bigwedge}_{i \leq n}{A_i}$ approximates ${\bigwedge}_{i \leq n}{A_i}$ and may be seen as a formal version of our heuristic. We will later restate the result in a less technical and more self-contained form.
\begin{therm}\label{thm:equaliser}
  Let $f, g : \bigwedge_{i \leq n} A_i \to B$. Given a homotopy $\widetilde{p}_n : (x : \widetilde{\bigwedge}_{i \leq n} A_i) \to f(\iota(x)) = g(\iota(x))$, there is a homotopy $p_n : (x : \bigwedge_{i \leq n} A_i) \to f(x) = g(x)$. Furthermore, $p_n$ satisfies:
  \begin{align}
  \label{eq:coh1}
  p_n(\star_\wedge) &= \widetilde{p}_n(\star_{\widetilde{\wedge}}) \\
  \label{eq:coh2}
  p_n \langle a_0,\dots,a_n\rangle &= \widetilde{p}_n \langle a_0,\dots,a_n\rangle
  \end{align}
\end{therm}
\begin{proof}
  We proceed by induction on $n$. When $n=0$, the map $\iota$ is simply the canonical equivalence between the cofibre of the unique pointed map $1 \to_\star A_0$ and ${A_0}$ itself,  and the theorem is trivial.

  Let $n \geq 1$ and let $f$, $g$ and $\widetilde{p}_n$ be as in the theorem statement. Let us, for clarity, rewrite the domain of $f$ and $g$ as the binary smash product $\left(\bigwedge_{i \leq {n-1}} A_{i}\right) \wedge A_n$. In order to construct \linebreak $p_n : (x : \left(\bigwedge_{i \leq {n-1}} A_{i}\right) \wedge A_n) \to f(x) = g(x)$, it suffices, by Lemma~\ref{lem:equaliser}, to define
  \begin{enumerate}[label=(\alph*)]
  \item \label{item:point-path} a path $p_n(\star_\wedge) : f(\star_\wedge) = g(\star_\wedge)$,
  \item \label{item:point-path-pairs} for $a_n : A_n$, homotopies $p_n\langle -, a_n \rangle : (x : \bigwedge_{i \leq {n-1}} A_{i}) \to f\langle x , y \rangle = g\langle x , y \rangle$, 
  \item \label{item:pushr} For each $a_n: A_n$, a filler of the square
    \[
\begin{tikzcd}[ampersand replacement=\&]
	{f(\star_\wedge)} \&\& {g(\star_\wedge)} \\
	{f\langle \star_\wedge,a_n\rangle} \&\& {g\langle \star_\wedge,a_n\rangle}
	\arrow["{p_n\langle \star_{\wedge} ,a_n\rangle}"', from=2-1, to=2-3]
	\arrow["{\ap{f}{\pushr{a_n}}}", from=2-1, to=1-1]
	\arrow["{\ap{g}{\pushr{a_n}}}"', from=2-3, to=1-3]
	\arrow["{p_n(\star_\wedge)}", from=1-1, to=1-3]
\end{tikzcd}
\]
  \item \label{item:pushl} For each $x: \bigwedge_{i \leq n-1}{A_i}$, a filler of the square
    \[
\begin{tikzcd}[ampersand replacement=\&]
	{f(\star_\wedge)} \&\& {g(\star_\wedge)} \\
	{f\langle x,\star_{A_n}\rangle} \&\& {g\langle x,\star_{A_n}\rangle}
	\arrow["{p_n\langle x ,\star_{A_n}\rangle}"', from=2-1, to=2-3]
	\arrow["{\ap{f}{\pushl{x}}}", from=2-1, to=1-1]
	\arrow["{\ap{g}{\pushl{x}}}"', from=2-3, to=1-3]
	\arrow["{p_n(\star_\wedge)}", from=1-1, to=1-3]
\end{tikzcd}
  \]
  \end{enumerate}
  For \ref{item:point-path}, we know that $\iota(\star_{\widetilde{\wedge}}) := \star_{\wedge}$ and thus $\widetilde{p}_n(\star_{\widetilde{\wedge}}) : f(\star_\wedge) = g(\star_\wedge)$. Consequently, we may simply set $p_n(\star_\wedge) := \widetilde{p}_n(\star_{\widetilde{\wedge}})$. Doing so, \eqref{eq:coh1} holds definitionally. For \ref{item:point-path-pairs}, we fix $a_n : A_n$. By the induction hypothesis, it suffices to construct $p_n\langle \iota(x),a_n\rangle : f\langle \iota(x) , y \rangle = g\langle \iota(x) , y \rangle$ for $x : \widetilde{\bigwedge}_{i \leq n-1}{A_i}$. We do this by the composite path
  \[
\begin{tikzcd}[ampersand replacement=\&]
	{f\langle \iota(x),a_n \rangle} \&\& {f(\iota(\uparrow (x,a_n)))} \&\& {g(\iota(\uparrow (x,a_n)))} \&\& {g\langle \iota(x),a_n\rangle}
	\arrow["{\ap{f}{\uparrow^{\mathsf{coh}}(x,a_n)}^{-1}}", from=1-1, to=1-3]
	\arrow["{\widetilde{p}_n(\uparrow (x,a_n))}", from=1-3, to=1-5]
	\arrow["{\ap{g}{\uparrow^{\mathsf{coh}}(x,a_n)}}", from=1-5, to=1-7]
\end{tikzcd}
\]
This completes the construction of $p_n \langle - , a_n\rangle $. Note that, since $p_n \langle - , a_n\rangle $ was constructed using the induction hypothesis, we also get the following from \eqref{eq:coh1}:
\begin{align*}
  p_n \langle \star_\wedge , a_n \rangle &= \ap{f}{\uparrow^{\mathsf{coh}}(\star_\wedge,a_n)}^{-1} \cdot {\widetilde{p}_n(\uparrow (\star_\wedge,a_n))} \cdot \ap{g}{\uparrow^{\mathsf{coh}}(\star_\wedge,a_n)} & \text{ by \eqref{eq:coh1}}\\
  &= \ap{f}{\pushr{a_n}} \cdot p_n(\star_\wedge) \cdot \ap{g}{\pushr{a_n}}^{-1} &\text{ by \eqref{eq:cohdef}}
\end{align*}
and the following from \eqref{eq:coh2}:
\begin{align*}
  p_n  \langle \vec{a}_{n-1} , a_n \rangle &= \ap{f}{\uparrow^{\mathsf{coh}}(\vec{a}_{n-1} ,a_n)}^{-1} \cdot {\widetilde{p}_n(\uparrow (\vec{a}_{n-1} ,a_n))} \cdot \ap{g}{\uparrow^{\mathsf{coh}}(\vec{a}_{n-1},a_n)} & \text{ by \eqref{eq:coh2}}\\
  &= \refl \cdot \widetilde{p}_n(\vec{a}_{n-1},a_n) \cdot \refl  &\text{ by \eqref{eq:cohdef-pairs}}\\
  &= \widetilde{p}_n(\vec{a}_{n-1},a_n)
\end{align*}
where $\vec{a}_{n-1}$ is short for $\langle a_0,\dots,a_{n-1} \rangle$. So, to summarise, we have the following two identities for $p\langle - , a_n \rangle$:
\begin{align}
  \label{eq:pid1}
  p_n \langle \star_\wedge , a_n \rangle = \ap{f}{\pushr{a_n}} \cdot p_n(\star_\wedge) \cdot \ap{g}{\pushr{a_n}}^{-1}
  \\
  \label{eq:pid2}
  p_n  \langle \langle a_0,\dots,a_{n-1}\rangle , a_n \rangle = \widetilde{p}_n  \langle a_0,\dots,a_{n-1} , a_n \rangle 
\end{align}
Note that, in particular, the latter identity verifies~\eqref{eq:coh2}. Let us now turn to \ref{item:pushr}. By \eqref{eq:pid1}, we simply need to fill the square
    \[
\begin{tikzcd}[ampersand replacement=\&]
	{f(\star_\wedge)} \&\&\&\&\&\& {g(\star_\wedge)} \\
	{f\langle \star_\wedge,a_n\rangle} \&\& \&\& \&\& {g\langle \star_\wedge,a_n\rangle}
	\arrow["\ap{f}{\pushr{a_n}} \cdot p_n(\star_\wedge) \cdot \ap{g}{\pushr{a_n}}^{-1}"', from=2-1, to=2-7]
	\arrow["{\ap{f}{\pushr{a_n}}}", from=2-1, to=1-1]
	\arrow["{\ap{g}{\pushr{a_n}}}"', from=2-7, to=1-7]
	\arrow["{p_n(\star_\wedge)}", from=1-1, to=1-7]
\end{tikzcd}
\]
which is entirely trivial by definition of path composition.

Finally, let us provide the data asked for in \ref{item:pushl}. Filling the square is equivalent to constructing, a homotopy $(x : \bigwedge_{i\leq n -1} A_i) \to \mathsf{left}(x) = \mathsf{right}(x)$ where $\mathsf{left},\mathsf{right} : \bigwedge_{i\leq n -1} A_i \to f(\star_\wedge) = g(\star_\wedge)$ are defined by
\begin{align*}
  \mathsf{left}(x) &:= p_n(\star_\wedge)\\
  \mathsf{right}(x) &:= \ap{f}{\pushl{x}}^{-1} \cdot p_n\langle x,\star_{A_n} \rangle \cdot \ap{g}{\pushl{x}}
\end{align*}
The codomain $f(\star_\wedge) = g(\star_\wedge)$ is homogeneous (with $\mathsf{left}(\star_\wedge)$ as basepoint) and thus Lemma~\ref{lem:evan-smash-unpointed} applies. Hence, we only need to show that
\[\mathsf{left}\langle a_0 , \dots, a_{n-1} \rangle = \mathsf{right}\langle a_0 , \dots, a_{n-1} \rangle \]
Let us rewrite $p_n\langle - , \star_{A_n} \rangle$ according to \eqref{eq:pid2} and transform the above back into square form. The problem is to fill the square
    \[
\begin{tikzcd}[ampersand replacement=\&]
	{f(\star_\wedge)} \&\&\&\& {g(\star_\wedge)} \\
	{f\langle \langle a_0,\dots,a_{n-1}\rangle,\star_{A_n}\rangle} \&\&\&\& {g\langle \langle a_0,\dots,a_{n-1}\rangle,\star_{A_n}\rangle}
	\arrow["{\widetilde{p}_n\langle a_0,\dots, a_{n-1} ,\star_{A_n}\rangle}"', from=2-1, to=2-5]
	\arrow["{\ap{f}{\pushll{\langle a_0,\dots,a_{n-1}\rangle}}}", from=2-1, to=1-1]
	\arrow["{\ap{g}{\pushll{\langle a_0,\dots,a_{n-1}\rangle}}}"', from=2-5, to=1-5]
	\arrow["{\widetilde{p}_n(\star_\wedge)}", from=1-1, to=1-5]
\end{tikzcd}
\]
Such a square is given precisely by $\ap{\widetilde{p}_n}{\pushop^{(2)}{(a_0,\dots,a_{n-1})}}$, and we are done.
\end{proof}
Let us restate the result in a more compact and self-contained form.
\begin{corollary}\label{cor:main}
  Let $f, g : \bigwedge_{i \leq n} A_i \to B$. If $f$ and $g$ agree on $ \iota : \widetilde{\bigwedge}_{i \leq n} A_i \to \bigwedge_{i \leq n} A_i$ , i.e.\ if $f \circ \iota = g \circ \iota$, then $f = g$.
\end{corollary}
For completeness, let us also state the analogous result for pointed maps.
\begin{corollary}\label{cor:main2}
  Let $f, g : \bigwedge_{i \leq n} A_i \to_\star B$. If $f \circ \iota = g \circ \iota$ as pointed maps, then we obtain an equality of pointed maps $f = g$.
\end{corollary}




\section{Conclusions and future work}
Let us summarise the key contributions of this paper. First and
foremost, we have shown in Sections~\ref{sec:assoc} through
\ref{sec:symmetric} that the smash product forms a 1-coherent
symmetric monoidal product on the wild category of pointed types (Theorem~\ref{thm:main}). This
fills a long-standing and rather troublesome gap in the HoTT
literature. In particular, it implies the correctness of previous
work which relies on this result---perhaps most notably, that that
of~\citep{Brunerie16}.

Secondly, we have presented a new method for reasoning about smash
products in HoTT. This was first done by introducing, in
Section~\ref{sec:heuristic}, an informal heuristic which was used
heavily leading up to Theorem~\ref{thm:main}. We later provided an
attempt at a formal version of the heuristic
in~Corollaries~\ref{cor:main}~and~\ref{cor:main2}. The message of
these results is simple: functions defined over smash products in HoTT
are not much harder to deal with than those defined over ordinary
Cartesian products. This contradicts a commonly held view that smash
products in HoTT become unworkable in higher dimensions.

In addition to salvaging already existing work relying on unproved
results about smash products, we also hope that the results presented
here will be useful in the further development of synthetic homotopy
theory in HoTT. In particular, Theorem~\ref{thm:main} is not the only
consequence
of~Propositions~\ref{prop:cup-comm}~and~\ref{prop:cup-assoc}. These
propositions capture two important properties of the equivalence
$\wedge_{n,m} : S^n \wedge S^m \simeq S^{n+m}$. For instance, \citet[Section
  VI.]{LM23} implicitly used the
special case when $n = m = 1$ in order to provide a simplified version of Brunerie's
proof of $\pi_4(S^3) \cong \bZ/2\bZ$. More generally, $\wedge_{n,m} :
S^n \wedge S^m \simeq S^{n+m}$ gives rise to the \emph{Whitehead
  product} \citep{whiteheadprod}. This operation was originally
defined in HoTT by~\citet[Definition 3.3.3.]{Brunerie16} using joins of spheres, but we can also
give it a rather compact construction in terms of $\wedge_{n,m}$: the Whitehead product $[f,g]$ of two
maps $f:S^{n+1} \to_\star A$ and $g: S^{m+1} \to_\star A$ can be
viewed as the map defined by the composition
\[ S^{n+m} \xrightarrow{\wedge_{n,m}^{-1}} (S^n \wedge S^m) \xrightarrow{\mathsf{comm}} \Omega (S^{n+1} \vee  S^{m+1}) \xrightarrow{\Omega(f\vee g)} \Omega A \]
where $\mathsf{comm}$ sends $\langle x,y\rangle$ to
$\sigma_{l}^{-1}(x) \cdot \sigma_{r}(y) \cdot \sigma_{l}(x) \cdot
\sigma_{r}^{-1}(y)$, where $\sigma : S^{k} \to \Omega S^{k+1}$ is
defined by $\sigma(x) = \merid{(x)}\cdot\merid{(\north)}^{-1}$ and the
subscripts $l$ and $r$ indicate in which component of $S^{n+1}
\vee S^{m+1}$ the loop takes place. This induces a map on homotopy
groups
\[\pi_{n+1}(A) \times \pi_{m+1}(A) \to \pi_{n+m}(\Omega(A)) \xrightarrow{\sim} \pi_{n+m+1}(A)\]
which, classically, turns $\pi_{\bullet}(A)$ into a graded quasi-Lie
algebra~\citep{UeharaMassey57}. This fact is still open in HoTT but we
conjecture that
Propositions~\ref{prop:cup-comm}~and~\ref{prop:cup-assoc} will play a
crucial role in a prospective proof. More concretely, as
$\wedge_{n,m}$ is used in the construction of this map,
Proposition~\ref{prop:cup-comm}~and~\ref{prop:cup-assoc} should, respectively, play
important roles in prospective proofs of graded
symmetry and the Jacobi identity.

On a related note, the heuristic presented
in~Section~\ref{sec:heuristic} and the related results in
Section~\ref{sec:equaliser} are reminiscent of results from the study
of polyhedral products~(see
e.g. \citet{bahri2019polyhedral,THERIAULT2018138}). This is a
field of mathematics which, very coarsely speaking, explores and
generalises the mediation between iterated wedge sums and iterated
Cartesian products. It is unclear whether this topic can be studied in
a meaningful way through the lens of HoTT, but exploring this would
certainly amount to an interesting continuation of the work we have
presented here.


\end{document}